\documentclass[]{amsart}

\usepackage{latexsym}
\usepackage{amssymb,amsmath,amsopn}
\usepackage[dvips]{graphicx}   
\usepackage{color,epsfig}      
\usepackage{url}

\newtheorem{theorem}{Theorem}[]
\newtheorem{lemma}{Lemma}[]

\theoremstyle{definition}

\newtheorem{remark}[]{Remark}[]

\newtheorem{note}[]{Note}[]

\DeclareMathOperator{\conv}{conv}

\DeclareMathOperator{\lin}{lin}

\title{On the volume growth of the hyperbolic regular $n$-simplex }
\author{\'Akos G.Horv\'ath}
\address {Department of Geometry \\
Budapest University of Technology and Economics\\
H-1521 Budapest\\
Hungary}
\email{ghorvath@math.bme.hu}
\date{Jan. 15, 2016}
\begin{document}


\begin{abstract}
In this paper we give lower and upper bounds for the volume growth of a regular hyperbolic simplex, namely for the ratio of the $n$-dimensional volume of a regular simplex and the $(n-1)$-dimensional volume of its facets. In addition to the methods of U. Haagerup and M. Munkholm we use a third volume form is based on the hyperbolic orthogonal coordinates of a body. In the case of the ideal, regular simplex our upper bound gives the best known upper bound. On the other hand, also in the ideal case our general lower bound, improved the best known one for $n=3$.
\end{abstract}

\maketitle

\section{Introduction}

\subsection{The problem}
The volume growth of the Euclidean regular $n$-simplex inscribed in the unit sphere can be calculated by the volume form
\begin{equation}\label{euclideangrowth}
V_n(S)=\frac{m}{n}V_{n-1}(F),
\end{equation}
where $V_n(S)$, $V_{n-1}(F)$ and $m$ means the $n$-dimensional volume of the simplex, the $(n-1)$-dimensional volume of its facet $F$ and the height of the simplex, respectively. Thus we have $V_{n}(S)/V_{n-1}(F)=(n+1)/n^2$. If $r$ is the radius of the circumscribed sphere of the regular simplex $rS$ then by similarity the ratio $V_{n}(rS)/V_{n-1}(rF)$ is equal to $r(n+1)/n^2$. In summary, the Euclidean "volume growth" is not an interesting quantity, it can be determined exactly and easily.

In contrast to the Euclidean case in hyperbolic $n$-space there is at least one kind of regular simplices which volume growth is very important. This is the regular, ideal simplex. In hyperbolic $n$-space a simplex is called \emph{regular} if any permutation of its vertices can be induced by an isometry of the space. A simplex \emph{ideal} if all the vertices are on the sphere at infinity. There is only, up to isometry, only one ideal regular simplex in hyperbolic $n$-space. It is known that the volume of a hyperbolic simplex is finite also if some of the vertices are on the sphere at infinity. Contrary the case of the plane, there is non-congruent ideal simplices arising an interesting question: Which one has maximal volume? The answer that \emph{ a simplex is of maximal volume if and only if it is ideal and regular} was conjectured (for any $n$) by Thurston in \cite{thurston}. In three space it was proved by Milnor (see in \cite{milnor} or \cite{milnor 2}). In higher dimension spaces this result was published by U. Haagerup and H. Munkholm in \cite{haagerupmunkholm}. Their prove based on the "volume growth" of an ideal, regular simplex,  Proposition 2 in \cite{haagerupmunkholm} says that $(n-2)/(n-1)^2\leq V_{n}(S)/V_{n-1}(F)\leq 1/(n-1)$. These bounds relies on an interplay between the volume forms of the Poincare's half-space model and the volume form of the Cayley-Klein's  projective model using also Gauss' divergence formula. Unfortunately, in hyperbolic $n$-space there is no formula analogous to (\ref{euclideangrowth}), hence to compare the $n$-volume of a regular simplex and the $(n-1)$-volume of its facets is a non-trivial exercise.

In recent decades the analytic investigations in hyperbolic space and even more the examination of computational methods of the $n$-dimensional hyperbolic volume appreciated. Without mentioning the exhaustive list I suggest to study the references \cite{bolyai, lobacsevszkij, kellerhals, cho, gho 1, gho 2, gho 3, gho 4, mednykh, murakami}. In this paper we give estimates for the volume growth of a regular hyperbolic simplex in the above detailed sense. In additions the methods of Haagerup and Munkholm we use a third volume form based on hyperbolic orthogonal coordinates. In the case of the ideal, regular simplex our upper bound gives the best known upper bound proved by Haagerup and Munkholm in 1981. On the other hand, also in the ideal case our general lower bound, improved the best known one for $n=3$.

\subsection{Notation}

We use the following notation in this paper:

\begin{itemize}
\item $\mathbb{R}$,$E^n$ and $H^n$: the set of real numbers, the Euclidean $n$-space and the hyperbolic $n$-spaces, respectively,
\item $|\cdot|$,$\rho(\cdot,\cdot)$: the Euclidean length and the hyperbolic distance function, respectively,
\item $V(\cdot)$: the volume function of the hyperbolic space,
\item $P^n$, $h:H^n\rightarrow P^n$: the half-space model of $H^n$ (Poincare's second model) and the standard mapping sending the hyperbolic space onto the model, respectively,
\item $H$: the boundary hyperplane of $P^n$,
\item $CK^n$ and $p:H^n \rightarrow CK^n$: the projective model of $H^n$ (Cayley-Klein's model) and the standard mapping on the hyperbolic space to the model
\item $x_1,\ldots,x_n$: the Euclidean coordinates of the embedding Euclidean space $E^n$ with respect to an orthonormed basis in $E^n$,
\item $S(n)=\conv\{E_1,\ldots,E_{n+1} \}$: the $n$-dimensional regular simplex inscribed in the unit sphere,
\item $\tau[n,t]=\conv\{p^{-1}(E_1),\ldots,p^{-1}(E_{n+1}) \}$: the hyperbolic regular $n$-simplex with hyperbolic circumradius $r(t)=\tanh^{-1}(\sin t )$,
\item $\tau_i[n,t]=\conv\{p^{-1}(E_1),\ldots,p^{-1}(E_{i-1}),p^{-1}(E_{i+1}),\ldots p^{-1}(E_{n+1}) \}$: a facet of $\tau[n,t]$, it is a hyperbolic, regular simplex of dimension $n-1$.
\item $r_k$: the hyperbolic radius of the circumscribed sphere of the $k$-dimensional faces of $\tau[n,t]$,
\item $d_k$: the hyperbolic distance of the circumcenter of $\tau[n,t]$ and a $(k-1)$-dimensional face of $\tau[n,t]$.
\end{itemize}

\section{The theorem}

Denote by $\tau[n,t]$ and $\tau_i[n,t]$ the $n$-dimensional regular simplex inscribed in the sphere of radius
$ r_n:=\tanh^{-1}(\sin t)$
and its $i^{th}$ facet, which is a regular simplex of dimension $(n-1)$ inscribed in a sphere of radius
$$
r_{n-1}:=\tanh^{-1}\left(\frac{\sin t\sqrt{1-\frac{1}{n^2}}}{\sqrt{1-\sin^2t\left(\frac{1}{n^2}\right)}}\right),
$$
respectively. In this paper we prove the following theorem:
\begin{theorem}\label{thmvolumegrowth}
We have the following two inequalities
\begin{equation}\label{ineqonvolumegrowthl}
(n+1)\left(\frac{1}{2}\right)^{n-1}\frac{\sqrt{1-\sin^2 t\frac{n-1}{2n}}}{\sqrt{1-\frac{n-1}{2n}}\sqrt{n^2-\sin^2 t}}(\cos t)\tanh^{-1}\left(\frac{\sin t\sqrt{1-\frac{n-1}{2n}}}{\sqrt{1-\sin^2 t\frac{n-1}{2n}}}\right)\leq \frac{V(\tau[n,t])}{V(\tau_{i}[n,t])}
\end{equation}
and
\begin{equation}\label{ineqonvolumegrowthu}
\frac{V(\tau[n,t])}{V(\tau_{i}[n,t])} \leq \frac{1}{n-1}\left(1-\left(\sqrt{\frac{n^2(1-\sin t)^2(1+\sin t)}{(n+\sin t)^2(1+\sin t)-(n^2-1)\sin ^2 t(1-\sin t)^2}}\right)^{n-1}\right).
\end{equation}
\end{theorem}

\begin{note}
For $n=3$ and $t=\pi/2$, (\ref{ineqonvolumegrowthl}) is stronger  than the left hand side of (3.1) in \cite{haagerupmunkholm}. In fact,
$$
\lim\limits_{t\rightarrow \pi/2}(\cos t)\tanh^{-1}\left(\frac{\sin t\sqrt{1-\frac{n-1}{2n}}}{\sqrt{1-\sin^2 t\frac{n-1}{2n}}}\right)=\lim\limits_{t\rightarrow \pi/2}(\cos t)\tanh^{-1}(\sin t)=1,
$$
hence we have for $n=3$ that
$$
\frac{n-2}{(n-1)^2}=\frac{1}{4}\leq \frac{4}{4\sqrt{8}}=\frac{n+1}{2^{n-1}\sqrt{n^2-1}}\leq \frac{1}{n-1}=\frac{1}{2}.
$$
Consequently, in dimension three our result is a generalization of the inequalities of \cite{haagerupmunkholm}.
\end{note}

\begin{note}
In the case of $t=\pi/2$, (\ref{ineqonvolumegrowthu}) gives the same bound as we saw in (3.1) in \cite{haagerupmunkholm}. Really,
$$
\frac{1}{n-1}\left(1-\left(\sqrt{\frac{n^2(1-\sin t)^2(1+\sin t)}{(n+\sin t)^2(1+\sin t)-(n^2-1)\sin ^2 t(1-\sin t)^2}}\right)^{n-1}\right)
$$
in $t=\pi/2$ is equal to $\frac{1}{n-1}$ (in all dimensions).

\end{note}

\section{The proof of the theorem}

The proof can be divided into three steps. 

The \emph{first step} uses that calculations which can be got from the Cayley-Klein (or projective) model. We determine certain metric properties of the regular hyperbolic simplex of circumradius $r_n$. 

In the \emph{second step} using hyperbolic orthogonal coordinates we prove the general lower bound.
 
The \emph{third step} contains the proof of the upper bound. In this section we use Poincare's half-space model. 

For the good readability these steps can be found in three subsection and the result is the union of the statements of Lemma \ref{lemlowerbound} and Lemma \ref{lemmaupperb}.

\subsection{Calculations in the projective model}

\begin{figure}[section]\label{projective}
\includegraphics[scale=1]{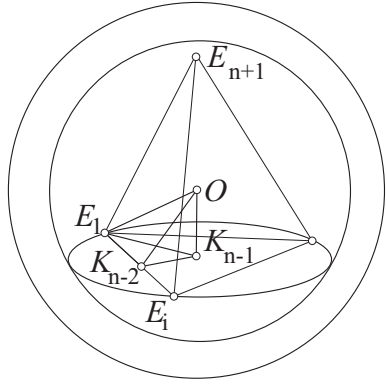}
\caption{Regular simplex in the projective model.}
\end{figure}

We consider the map $p:H^n\rightarrow CK^n$ sending the regular ideal simplex $\tau[n,\pi/2]$ into the Euclidean regular simplex $p(\tau[n,\pi/2])$ inscribed in the unit sphere. (Clearly, the map $p$ can be given concretely but it is not important for our purpose.) For every $t\in [0,\pi/2]$ $p(\tau[n,t])$ denotes the regular simplex concentric with $p(\tau[n,\pi/2])$ and with circumradius $r=\sin t$. For simplicity $S(n)$ denotes the regular simplex $p(\tau[n,\pi/2])$. Observe that the volume of $\tau[n,t]$ can be calculated from the volume form of $CK^n$ (see in \cite{gho 1})
$$
dV=\left(1-r^2\right)^{-(n+1)/2}\mathrm{d}x_1 \ldots \mathrm{d}x_n, \quad r^2:=\sum\limits_{i=1}^nx_i^2,
$$
and we get that
\begin{equation}\label{ckvolume}
V(\tau[n,t])=\int\limits_{(\sin t)S(n)}\left(1-r^2\right)^{-\frac{n+1}{2}}\mathrm{d}r=\sin^n t\int\limits_{S(n)}\left(1-(\sin^2t) r^2\right)^{-\frac{n+1}{2}}\mathrm{d}r.
\end{equation}
By spherical symmetry we can choose a coordinate system such a way, that $E_{n+1}=(0,\ldots,0,\sin t)^T$. Denote by $O$ and $K_{n-1}$ the center of the simplex $S(n)$ and the center its facet $\conv\{E_1,\ldots,E_n\}$, respectively. Then $|OK_{n-1}|=\sin t/n$ and the hyperbolic distance of the corresponding points is
$$
\rho\left(p^{-1}(O),p^{-1}(K_{n-1})\right)=\frac{1}{2}\ln \frac{n+\sin t}{n-\sin t}= \tanh^{-1}(\sin t).
$$
Immediate calculation with cross-ratio gives the circumradius of the facets $\tau_i[n,t]$ ($i=1,\ldots, n+1$) of  $\tau[n,t]$:
$$
\rho\left(p^{-1}(E_1),p^{-1}(K_{n-1})\right)=\ln \frac{\sqrt{n^2-\sin^2t}+\sin t\sqrt{n^2-1}}{n\cos t}.
$$
Denote by $K_i$ the center of the face $\conv\{E_1,E_2,\ldots,E_{i+1}\}$ for $i=1,\ldots,n-1$. Then the \emph{fundamental orthosceme} $\mathcal{O}$ of the simplex $\tau[n,t]$ is the convex hull of  the points $\{E_1,K_1,K_2,\ldots, K_{n-1},O\}$. Clearly $\tau[n,t]$ is disjoint union of $(n+1)!$ congruent copies of $\mathcal{O}$ hence $V(\tau[n,t])=(n+1)!V(\mathcal{O})$. The successive edge lengths of  $\mathcal{O}$ can be determined by induction. As we saw
$$
r_n:=\rho\left(p^{-1}(O),p^{-1}(E_{1})\right)=\tanh^{-1}(\sin t) \quad \mbox{ and } \quad d_{n}:=\rho\left(p^{-1}(O),p^{-1}(K_{n-1})\right)=\tanh^{-1}\left(\frac{\sin t}{n}\right)
$$
moreover
$$
r_{n-1}:=\rho\left(p^{-1}(K_{n-1}),p^{-1}(E_{1})\right)=\ln \frac{\sqrt{n^2-\sin^2t}+\sin t\sqrt{n^2-1}}{n\cos t}=
$$
$$
=\frac{1}{2}\ln\frac{\sqrt{1-\sin^2t\left(\frac{1}{n^2}\right)}+\sin t\sqrt{1-\frac{1}{n^2}}
}{\sqrt{1-\sin^2t\left(\frac{1}{n^2}\right)}-\sin t\sqrt{1-\frac{1}{n^2}}}=\tanh^{-1}\left(\frac{\sin t\sqrt{1-\frac{1}{n^2}}}{\sqrt{1-\sin^2t\left(\frac{1}{n^2}\right)}}\right)
$$
and
$$
d_{n-1}:=\rho\left(p^{-1}(K_{n-1}),p^{-1}(K_{n-2})\right)=\frac{1}{2}\ln\frac{\sqrt{1-\sin^2t\left(\frac{1}{n^2}\right)}+\frac{1}{(n-1)}\sin t\sqrt{1-\frac{1}{n^2}}
}{\sqrt{1-\sin^2t\left(\frac{1}{n^2}\right)}-\frac{1}{(n-1)}\sin t\sqrt{1-\frac{1}{n^2}}}=
$$
$$
=\tanh^{-1}\left(\frac{\sin t\sqrt{1-\frac{1}{n^2}}}{(n-1)\sqrt{1-\sin^2t\left(\frac{1}{n^2}\right)}}\right)
$$
it can be proved by induction that
\begin{equation}\label{circumradii}
r_{n-k}:=\rho\left(p^{-1}(K_{n-k}),p^{-1}(E_{1})\right)=\tanh^{-1}\left(\frac{\sin t\sqrt{1-\frac{k}{n(n-k+1)}}}{\sqrt{1-\sin^2t\left(\frac{k}{n(n-k+1)}\right)}}\right)
\end{equation}
holds for $k=1, \ldots, (n-1)$ and
\begin{equation}\label{edgelengths}
d_{n-k}:=\rho\left(p^{-1}(K_{n-k}),p^{-1}(K_{n-k-1})\right)=\tanh^{-1}\left(\frac{\sin t\sqrt{1-\frac{k}{n(n-k+1)}}}{(n-k)\sqrt{1-\sin^2t\left(\frac{k}{n(n-k+1)}\right)}}\right)
\end{equation}
holds also for $k=1, \ldots, (n-1)$ if $K_0$ means $E_1$.

With respect to an $(n-1)$-dimensional projective model we can determine the volume of a facet of the regular $n$-simplex with circumradius $r_n$. In fact, the circumradius of its facets is $r_{n-1}$ hence using the formula (\ref{ckvolume}) we get
\begin{equation}\label{ckvolumeofthefacets}
V(\tau_i[n,t])=\int\limits_{(\tanh {r_{n-1})S(n-1)}}\left(1-r^2\right)^{-\frac{n}{2}}\mathrm{d}r=\int\limits_{\frac{\sin t\sqrt{1-\frac{1}{n^2}}}{\sqrt{1-\sin^2t\left(\frac{1}{n^2}\right)}}S(n-1)}\left(\frac{1}{\sqrt{1-r^2}}\right)^{n}\mathrm{d}r.
\end{equation}

\subsection{Hyperbolic orthogonal coordinates and the lower bound}

Take a hyperbolic orthogonal coordinate system with origin $p^{-1}(E_1)$ on the following way. Let the $x_n$ axis be the line through the points $p^{-1}(E_1)$ and $p^{-1}(E_{2})$ containing the point $p^{-1}(K_1)$, too. Let the $x_{1}$ axis the line of the hyperbolic 2-plane containing the points $p^{-1}(E_1)$, $p^{-1}(K_1)$ and $p^{-1}(K_2)$; and directed such that the coordinates of the points of the rectangular triangle $p^{-1}(E_1), p^{-1}(K_1),p^{-1}(K_2)$ will be non-negative real numbers. Continue this process and choose the axis $x_k$  such that the $k-1$ subspace is spanned by the axis $\lin\{x_n,x_1,x_2,\ldots,x_{k}\}$ contains the $(k-1)$-dimensional orthosceme is spanned by the points $\{p^{-1}(E_1),p^{-1}(K_1), \ldots, p^{-1}(K_{n-1})\}$ in its positive orthant (hyperoctant). Then we can apply the following volume form (see paragraph 3.3.3 in \cite{gho 1}) 
$$
V(\mathcal{O})=\int\limits_{0}^{\alpha_n}\int\limits_{0}^{\alpha_1}\cdots  \int\limits_{0}^{\alpha_{n-1}}
(\cosh^{n-1}x_{n-1})(\cosh^{n-2}  x_{n-2})\cdots (\cosh  x_{1})
\mathrm{d}x_{n-1}\cdots \mathrm{d}x_{1}\mathrm{d}x_n,
$$
where
\begin{eqnarray*}
\alpha_{n-1}& = &\tanh^{-1}\left(\frac{ \tanh d_{n}}{\sinh d_{n-1}}\sinh x_{n-2}\right) \\
\alpha_{n-2}& = & \tanh^{-1}\left(\frac{ \tanh d_{n-1}}{\sinh d_{n-2}}\sinh x_{n-3}\right)\\
& \vdots & \\
\alpha_2 &= &\tanh^{-1}\left(\frac{ \tanh d_{3}}{\sinh d_2}\sinh x_1\right)\\
\alpha_1 & = & \tanh^{-1}\left(\frac{ \tanh d_{2}}{\sinh d_1}\sinh x_n\right)\\
\alpha_n & = & d_1.
\end{eqnarray*}
Observe that the volume of the facet $\mathcal{O}_{n-1}:=\conv\{p^{-1}(E_1),p^{-1}(K_1),p^{-1}(K_2)\ldots p^{-1}(K_{n-1})\}$ is
$$
V(\mathcal{O}_{n-1})=\int\limits_{0}^{\alpha_n}\int\limits_{0}^{\alpha_1}\cdots  \int\limits_{0}^{\alpha_{n-2}}(\cosh^{n-2}  x_{n-2})\cdots (\cosh  x_{1})
\mathrm{d}x_{n-2}\cdots \mathrm{d}x_{1}\mathrm{d}x_n,
$$
with the same $\alpha_i$'s. Consequently these formulas leads to a connection between $V(\mathcal{O})$ and $V(\mathcal{O}_{n-1})$ and between $V(\tau[n,t])=(n+1)!V(\mathcal{O})$ and $V(\tau_{i}[n,t])=n!V(\mathcal{O}_{n-1})$. 
In the proof we use Chebyshev's integral inequality (see in \cite{mitrinovic}), saying that if $f,g:[a,b]\rightarrow \mathbb{R}$ have the same monotony then
$$
\int\limits_a^bf(x)g(x)\mathrm{d}x\geq \frac{1}{b-a}\int\limits_a^bf(x)\mathrm{d}x\int\limits_a^bg(x)\mathrm{d}x,
$$
and if $f; g$ have opposite monotony, then the inequality should be reversed.
\begin{lemma}\label{lemlowerbound}
For $n\geq 3$ we have
\begin{equation}\label{lowerbound}
V(\tau[n,t])\geq (n+1)\left(\frac{1}{2}\right)^{n-1}\frac{\sqrt{1-\sin^2 t\frac{n-1}{2n}}}{\sqrt{1-\frac{n-1}{2n}}\sqrt{n^2-\sin^2 t}}(\cos t)\tanh^{-1}\left(\frac{\sin t\sqrt{1-\frac{n-1}{2n}}}{1-\sin^2 t\frac{n-1}{2n}}\right)V(\tau_{i}[n,t])
\end{equation}
\end{lemma}

\begin{proof}
We compare the following two integrals
$$
V(\mathcal{O})=\int\limits_{0}^{\alpha_n}\int\limits_{0}^{\alpha_1}\cdots  \int\limits_{0}^{\alpha_{n-2}}\left(\int\limits_{0}^{\alpha_{n-1}}
\cosh^{n-1}x_{n-1}\mathrm{d}x_{n-1}\right)\cosh^{n-2}  x_{n-2}\cdots \cosh  x_{1}\mathrm{d}x_{n-2}
\cdots \mathrm{d}x_{1}\mathrm{d}x_n,
$$
and
$$
V(\mathcal{O}_{n-1})=\int\limits_{0}^{\alpha_n}\int\limits_{0}^{\alpha_1}\cdots  \int\limits_{0}^{\alpha_{n-2}}\cosh^{n-2}  x_{n-2}\cdots \cosh  x_{1}
\mathrm{d}x_{n-2}\cdots \mathrm{d}x_{1}\mathrm{d}x_n.
$$
Since
\begin{eqnarray*}
\cos^{2k}(x) & = & \frac{1}{4^k}\binom{2k}{k}+\frac{1}{2^{2k-1}} \sum\limits_{l=0}^{k-1}\binom{2k}{l}\cos(2(k-l)x) \\
\cos^{2k+1}x & = & \frac{1}{4^{k}}\sum\limits_{l=0}^{k}\binom{2k+1}{l}\cos((2(k-l)+1)x)
\end{eqnarray*}
then using the identity $\cosh x=\cos (ix)$ we have
\begin{eqnarray*}
\cosh^{2k} x &= &\frac{1}{4^k}\binom{2k}{k}+\frac{1}{2^{2k-1}} \sum\limits_{l=0}^{k-1}\binom{2k}{l}\cosh(2(k-l)x)\\
\cosh^{2k+1}x & = & \frac{1}{4^{k}}\sum\limits_{l=0}^{k}\binom{2k+1}{l}\cosh((2(k-l)+1)x).
\end{eqnarray*}
Let denote $F_{n-1}(x)$ the antiderivative of the functions $\cosh^{n-1}x$.
Then we have
\begin{equation}\label{antiderivate}
F_{n-1}(x)=\left\{\begin{array}{lcl}
\frac{1}{4^k}\binom{2k}{k}x+\frac{1}{2^{2k-1}} \sum\limits_{l=0}^{k-1}\binom{2k}{l}\frac{1}{2(k-l)}\sinh(2(k-l)x) & \mbox{ if } & n=2k \\
\frac{1}{4^{k}}\sum\limits_{l=0}^{k}\binom{2k+1}{l}\frac{1}{(2(k-l)+1)}\sinh((2(k-l)+1)x) & \mbox{ if } & n=2k+1.
\end{array}
\right.
\end{equation}
Since $F_{n-1}(0)=0$ we get 
$$
f_{n-1}(x_{n-2}):=\int\limits_{0}^{\alpha_{n-1}}\cosh^{n-1}(x_{n-1})\mathrm{d}x_{n-1}=
$$
$$
=\left\{\begin{array}{lcl}
\frac{1}{4^k}\binom{2k}{k}\alpha_{n-1}(x_{n-2})+\frac{1}{2^{2k-1}} \sum\limits_{l=0}^{k-1}\binom{2k}{l}\frac{1}{2(k-l)}\sinh(2(k-l)\alpha_{n-1}(x_{n-2})) & \mbox{ if } & n=2k \\
\frac{1}{4^{k}}\sum\limits_{l=0}^{k}\binom{2k+1}{l}\frac{1}{(2(k-l)+1)}\sinh((2(k-l)+1)\alpha_{n-1}(x_{n-2})) & \mbox{ if } & n=2k+1
\end{array}
\right.
$$
where
$$
\alpha_{n-1}(x_{n-2})=\tanh^{-1}\left(\frac{ \tanh d_{n}}{\sinh d_{n-1}}\sinh x_{n-2}\right)=\tanh^{-1}\left(\sqrt{\frac{n-1}{n+1}}\sqrt{1-\frac{2\sin^2t}{n(n-1)}}\sinh x_{n-2}\right).
$$
Without loss of generality we can assume that $n=2k+1$, the other case when $n=2k$ can be proved analogously leading to the same result.
Since $f_{n-1}(x_{n-2})$ is strictly increasing and $\sinh x,\tanh^{-1}(x)\geq x$ for all positive $x$, we have
$$
f_{n-1}(x_{n-2})\geq \frac{1}{4^{k}}\sum\limits_{l=0}^{k}\binom{2k+1}{l}\frac{1}{(2(k-l)+1)}(2(k-l)+1)\alpha_{n-1}(x_{n-2}))=\alpha_{n-1}(x_{n-2}))=
$$
$$
=\tanh^{-1}\left(\frac{ \tanh d_{n}}{\sinh d_{n-1}}\sinh x_{n-2}\right)\geq \frac{ \tanh d_{n}}{\sinh d_{n-1}}\sinh x_{n-2},
$$
implying that
$$
V(\mathcal{O})= \int\limits_{0}^{\alpha_n}\int\limits_{0}^{\alpha_1}\cdots  \int\limits_{0}^{\alpha_{n-2}}\left(\int\limits_{0}^{\alpha_{n-1}}
\cosh^{n-1}x_{n-1}\mathrm{d}x_{n-1}\right)\cosh^{n-2}  x_{n-2}\cdots \cosh  x_{1}\mathrm{d}x_{n-2}
\cdots \mathrm{d}x_{1}\mathrm{d}x_n\geq
$$
$$
\geq \frac{ \tanh d_{n}}{\sinh d_{n-1}}\int\limits_{0}^{\alpha_n}\int\limits_{0}^{\alpha_1}\cdots  \int\limits_{0}^{\alpha_{n-2}}\sinh x_{n-2}\cosh^{n-2}  x_{n-2}\mathrm{d}x_{n-2}\cosh^{n-3}  x_{n-3}\cdots \cosh  x_{1}\mathrm{d}x_{n-3}
\cdots \mathrm{d}x_{1}\mathrm{d}x_n.
$$
Using the Chebysev's inequality we get that
$$
\int\limits_{0}^{\alpha_{n-2}}\sinh x_{n-2}\cosh^{n-2}  x_{n-2}\mathrm{d}x_{n-2}\geq \frac{\cosh \alpha_{n-2}-1}{\alpha_{n-2}}\int\limits_{0}^{\alpha_{n-2}}\cosh^{n-2}  x_{n-2}\mathrm{d}x_{n-2}\geq
$$
$$
\geq\frac{ \alpha_{n-2}}{2}\int\limits_{0}^{\alpha_{n-2}}\cosh^{n-2}  x_{n-2}\mathrm{d}x_{n-2}\geq \frac{1}{2}\frac{ \tanh d_{n-1}}{\sinh d_{n-2}}\sinh x_{n-3}\int\limits_{0}^{\alpha_{n-2}}\cosh^{n-2}  x_{n-2}\mathrm{d}x_{n-2}.
$$
So
$$
V(\mathcal{O})\geq \frac{1}{2}\frac{ \tanh d_{n}}{\sinh d_{n-1}}\frac{ \tanh d_{n-1}}{\sinh d_{n-2}}\int\limits_{0}^{\alpha_n}\int\limits_{0}^{\alpha_1}\cdots  \int\limits_{0}^{\alpha_{n-3}}\sinh x_{n-3}\left(\int\limits_{0}^{\alpha_{n-2}}\cosh^{n-2}  x_{n-2}\mathrm{d}x_{n-2}\right)\times
$$
$$
\times\cosh^{n-3}  x_{n-3}\mathrm{d}x_{n-3}\cosh^{n-4}  x_{n-4}\cdots \cosh  x_{1}\mathrm{d}x_{n-4}
\cdots \mathrm{d}x_{1}\mathrm{d}x_n
$$
From ($\ref{antiderivate}$) can be seen easily that the function
$$
\left(\int\limits_{0}^{\alpha_{n-2}}\cosh^{n-2}  x_{n-2}\mathrm{d}x_{n-2}\right)\cosh^{n-3}  x_{n-3}
$$
holds the assumption of Chebysev's  inequality and hence we can continue the process.
$$
V(\mathcal{O})\geq \left(\frac{1}{2}\right)^2\frac{ \tanh d_{n}\tanh d_{n-1}\tanh d_{n-2}}{\sinh d_{n-1}\sinh d_{n-2}\sinh d_{n-3}}\int\limits_{0}^{\alpha_n}\int\limits_{0}^{\alpha_1}\cdots  \int\limits_{0}^{\alpha_{n-4}}\sinh x_{n-4}\times
$$
$$
\times\left(\int\limits_{0}^{\alpha_{n-3}}\int\limits_{0}^{\alpha_{n-2}}\cosh^{n-2}  x_{n-2}\cosh^{n-3}  x_{n-3}\mathrm{d}x_{n-2}\mathrm{d}x_{n-3}\right)\cosh^{n-4}  x_{n-4}\mathrm{d}x_{n-4}\times
$$
$$
\times\cosh^{n-5}  x_{n-5}\cdots \cosh  x_{1}\mathrm{d}x_{n-5}
\cdots \mathrm{d}x_{1}\mathrm{d}x_n
$$
and so on... By induction we get the inequality
$$
\frac{1}{(n+1)!}V(\tau[n,t])=V(\mathcal{O})\geq \left(\frac{1}{2}\right)^{n-1}\frac{\left(\tanh d_{n}\cdots \tanh d_{2}\right)d_1}{\sinh d_{n-1}\cdots \sinh d_1}V(\mathcal{O}_{n-1})=
$$
$$
=\left(\frac{1}{2}\right)^{n-1}\frac{\left(\sinh d_{n}\right)d_1}{\cosh d_{n}\cdots \cosh d_2\sinh d_1}V(\mathcal{O}_{n-1})=\left(\frac{1}{2}\right)^{n-1}\frac{\left(\sinh d_{n}\cosh d_1\right)d_1}{\cosh r_{n}\sinh d_1}\frac{V(\tau_{i}[n,t])}{n!}.
$$
But
$$
\sinh d_n=\frac{\sin t}{\sqrt{n^2-\sin^2 t}}, \quad \cosh r_n=\frac{1}{\cos t},
$$
$$
d_1=\tanh^{-1}\left(\frac{\sin t\sqrt{1-\frac{n-1}{2n}}}{1-\sin^2 t\frac{n-1}{2n}}\right), \quad
\sinh d_1=\frac{\sin t\sqrt{1-\frac{n-1}{2n}}}{\cos t} \quad  \cosh d_1=\frac{\sqrt{1-\sin^2 t\frac{n-1}{2n}}}{\cos t},
$$
we get the inequality
$$
V(\tau[n,t])\geq (n+1)\left(\frac{1}{2}\right)^{n-1}\frac{\sqrt{1-\sin^2 t\frac{n-1}{2n}}}{\sqrt{1-\frac{n-1}{2n}}\sqrt{n^2-\sin^2 t}}(\cos t)\tanh^{-1}\left(\frac{\sin t\sqrt{1-\frac{n-1}{2n}}}{1-\sin^2 t\frac{n-1}{2n}}\right)V(\tau_{i}[n,t]),
$$
as we stated.
\end{proof}

\subsection{Half-space model and the upper bound}

\begin{figure}[section]\label{halfspace}
\includegraphics[scale=1]{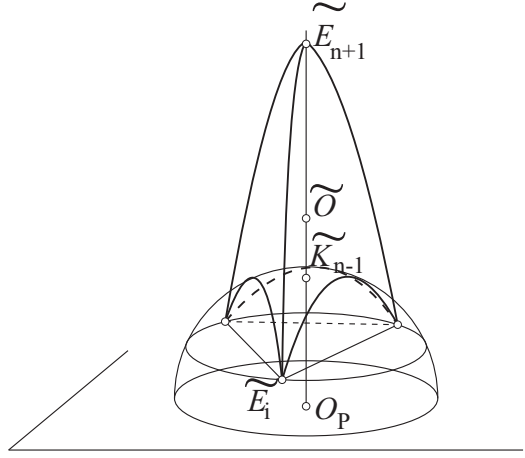}
\caption{Regular simplex in the half-space model.}
\end{figure}

In the half-space model we consider an Euclidean orthogonal coordinate system with origin $O_P$ lying on the boundary (horizontal) hyperplane $H$. For the sake of simplicity denote by $\widetilde{P}$ the point $h(p^{-1}(P))\in P^n$ where $P\in CK^n$ is a point of the projective model. Assume that $O_P, \widetilde{K_{n-1}}, \widetilde{O}$ and $\widetilde{E_{n+1}}$ lying on a vertical line, which is the $x_n$ axis of our coordinate system (see in Fig. 2). From the regularity of $\tau[n,t]$ follows that the Euclidean simplex $\conv\{\widetilde{E_1},\ldots ,\widetilde{E_n}\}$ of dimension $(n-1)$ is also regular in Euclidean sense and lying in a horizontal (parallel to $H$) hyperplane.

Set $|O_P\widetilde{K_{n-1}}|=1$ then $|O_P\widetilde{O}|=\sqrt{\frac{n+\sin t}{n-\sin t}}$ and $|O_P\widetilde{E_{n+1}}|=\sqrt{\frac{(n+\sin t)(1+\sin t)}{(n-\sin t)(1-\sin t)}}$.
If $\alpha$ is the Euclidean angle $\widetilde{E_i}O_P\widetilde{O}\angle$ $(1\leq i\leq n-1)$ then we have:
$$
\rho(E_1\widetilde{K})=\ln \frac{1+\sin \alpha}{\cos \alpha}
$$
and the equation between $\alpha$ and $t$ is
$$
\frac{\sqrt{n^2-\sin^2t}+\sin t\sqrt{n^2-1}}{n\cos t}=\frac{1+\sin \alpha}{\cos \alpha}
$$
implying that
$$
\sin t=\frac{n\sin \alpha}{\sqrt{n^2-\cos^2 \alpha}} \quad \mbox{ or }\quad \sin\alpha= \frac{\sqrt{n^2-1}\sin t}{\sqrt{n^2-\sin^2 t}}=\frac{\sin t\sqrt{1-\frac{1}{n^2}}}{\sqrt{1-\frac{\sin^2 t}{n^2}}}.
$$
The Euclidean coordinates of the points $\widetilde{E_i}=(x_j^i)$, $j=1,\ldots, n$, $i=1,\ldots, n+1$ hold the equalities:
$$
\sum\limits_{j=1}^{n-1}(x_j^i)^2=\sin^2\alpha =\frac{(n^2-1)\sin^2 t}{n^2-\sin ^2t} \quad \mbox{for all} \quad i=1\ldots n,
$$
$$
(x_n^i)^2=\cos^2\alpha =\frac{n^2\cos^2t}{n^2-\sin^2 t} \quad \mbox{for all} \quad i=1\ldots n
$$
$$
\sum\limits_{j=1}^{n-1}(x_j^{n+1})^2=0, \quad (x_n^{n+1})^2=\frac{(n+\sin t)(1+\sin t)}{(n-\sin t)(1-\sin t)}.
$$
Let denote by $C_i=\left(y_1^i,\ldots ,y_{n-1}^i,0\right)^T$ and $\gamma_i$ the center and the radius of the sphere through the points $\{\widetilde{E_1},\ldots,\widetilde{E_{i-1}},\widetilde{E_{i+1}},\ldots, \widetilde{E_{n+1}}\}$. This sphere contains a facet of $h(\tau[n,t])$. Hence for all $i=1,\ldots ,n $ we have  $(n-1)$ equalities holding for the coordinates above. These are
$$
\sum\limits_{j=1}^{n-1}x_j^ky_j^i=\frac{-(n+1)\sin t}{(n-\sin t)(1-\sin t)}=:c \quad \mbox{ where } k\neq i, \quad k=1,\ldots, n.
$$
Introduce the notation $X^i=:[x_j^k]_{j=1,\ldots,n-1}^{k=1,\ldots,n; k\ne i} \in \mathbb{R}^{(n-1)\times (n-1)}$ and $y^i=[y_1^i,\ldots ,y_{n-1}^i]^T$, respectively. Then we get
$$
(y^i)^TX^i=[c,\ldots,c] \quad \mbox{ or equivalently } (X^i)^Ty^i=[c,\ldots,c]^T=c[1,\ldots,1]^T.
$$
Since $\conv\{\widetilde{E_1},\ldots ,\widetilde{E_n}\}$ is a regular simplex we get
$$
R(n-1):=(X^i)^T(X^i)=\frac{(n^2-1)\sin^2 t}{n^2-\sin^2 t}\left(\begin{array}{ccccc}
                                                   1 & -\frac{1}{n-1} & \ldots & -\frac{1}{n-1} & -\frac{1}{n-1} \\
                                                   -\frac{1}{n-1} & 1 & -\frac{1}{n-1} & \ldots & -\frac{1}{n-1} \\
                                                   \vdots & \vdots & \vdots & \vdots & \vdots \\
                                                   -\frac{1}{n-1} & \ldots & -\frac{1}{n-1} & 1 & -\frac{1}{n-1} \\
                                                   -\frac{1}{n-1} & -\frac{1}{n-1} & \ldots  & -\frac{1}{n-1} & 1 \\
                                                 \end{array}
                                               \right).
$$
It follows that
$$
(X^i)^T=R(n-1)(X^i)^{-1} \quad \mbox{ and so } ((X^i)^T)^{-1}=(X^i)(R(n-1))^{-1}.
$$
Since the inverse of $R(n-1)$ is equal to
$$
(R(n-1))^{-1}=\frac{n^2-\sin^2 t}{n(n+1)\sin^2 t}\left(
                                                  \begin{array}{ccccc}
                                                    2 & 1 & 1 & \ldots & 1 \\
                                                    1 & 2 & 1 & \ldots & 1 \\
                                                    \vdots & \vdots & \vdots & \vdots & \vdots \\
                                                    1 & \ldots & 1 & 2 & 1 \\
                                                    1 & 1 & \ldots & 1 & 2 \\
                                                  \end{array}
                                                \right)
$$
we get
$$
y^i=X^i(R(n-1))^{-1}[c,\ldots,c]^T=-\frac{(n+\sin t)}{n\sin t(1-\sin t)}X^i\left(
                                                  \begin{array}{ccccc}
                                                    2 & 1 & 1 & \ldots & 1 \\
                                                    1 & 2 & 1 & \ldots & 1 \\
                                                    \vdots & \vdots & \vdots & \vdots & \vdots \\
                                                    1 & \ldots & 1 & 2 & 1 \\
                                                    1 & 1 & \ldots & 1 & 2 \\
                                                  \end{array}\right)\left(
                                                                      \begin{array}{c}
                                                                        1\\
                                                                        1\\
                                                                        \vdots \\
                                                                        1\\
                                                                        1 \\
                                                                      \end{array}
                                                                    \right)=
$$
$$
=-\frac{(n+\sin t)}{\sin t(1-\sin t)}X^i\left(\begin{array}{c}
                                                                        1 \\
                                                                        1\\
                                                                        \vdots \\
                                                                        1 \\
                                                                        1 \\
                                                                      \end{array}
                                                                    \right)=-\frac{(n+\sin t)}{\sin t(1-\sin t)}\left(
                                                                      \begin{array}{c}
                                                                        \sum\limits_{\substack{k=1 \\  k\ne i}}^{n}x_1^k\\
                                                                        \sum\limits_{\substack{k=1 \\  k\ne i}}^{n}x_2^k\\
                                                                        \vdots \\
                                                                        \sum\limits_{\substack{k=1 \\  k\ne i}}^{n}x_{n-2}^k \\
                                                                        \sum\limits_{\substack{k=1 \\  k\ne i}}^{n}x_{n-1}^k\\
                                                                      \end{array}\right)=-\frac{(n+\sin t)}{\sin t(1-\sin t)}\sum\limits_{\substack{k=1 \\  k\ne i}}^{n}v_k,
$$
where $v_k:=[x_1^k,\ldots,x_{n-1}^k]^T$. Since $\sum\limits_{k=1}^{n}v_k=0$, we get that
$$
y^i=\frac{(n+\sin t)}{\sin t(1-\sin t)}v_i.
$$
Hence
$$
|y^i|^2=\frac{(n+\sin t)^2}{\sin^2 t(1-\sin t)^2}|v_i|^2=\frac{(n^2-1)(n+\sin t)}{(n-\sin t)(1-\sin t)^2},
$$
implying that
$$
\gamma_i^2=\left(\frac{n+\sin t}{1-\sin t}\right)^2.
$$
Let denote by $\epsilon[n-1,t]=\conv\{F_1,\ldots ,F_n\}$ where $F_i$ is the orthogonal projection of $\widetilde{E_i}$ to $H$. $\epsilon[n-1,t]$ can be dissected into $n$ congruent simplices with a common vertex $O_P$. Denote by $\epsilon_i[n-1,t]$ these simplices, concretely, we set $\epsilon_i[n-1,t]:=\{\sum_{j\ne i}\alpha_j v_j \quad | \quad \alpha_j\geq 0, \sum_{j\ne i} \alpha_j\leq 1\}$. If a point $z\in P^n$ is in  $h(\tau[n,t])$,  then it is on the form $z=v+z_n(v)e_n$, where $v\in \epsilon([n-1,t])$.  If $z\in h(\tau[n,t])$ with $v\in  \epsilon_n[n-1,t]$ then $\|z-C_n\|^2\leq \gamma_n^2$ hence we have that
$$
\sum\limits_{i=1}^{n-1}\left(\sum\limits_{j=1}^{n-1}\alpha_j x_i^j-y_i^n\right)^2+z_n^2\leq (\gamma_n)^2.
$$
From this inequality we get
$$
z_n^2\leq \left(\frac{n+\sin t}{1-\sin t}\right)^2- \sum\limits_{i=1}^{n-1}\left(\sum\limits_{j=1}^{n-1}\alpha_j x_i^j-\frac{(n+\sin t)}{\sin t(1-\sin t)}x_i^n\right)^2=
$$
$$
=\left(\frac{n+\sin t}{1-\sin t}\right)^2- \sum\limits_{i=1}^{n-1}\left(\sum\limits_{j=1}^{n-1}\left(\alpha_j+\frac{(n+\sin t)}{\sin t(1-\sin t)}\right)x_i^j\right)^2=
$$
$$
=\left(\frac{n+\sin t}{1-\sin t}\right)^2- \sum\limits_{i=1}^{n-1}\sum\limits_{j,k=1}^{n-1}\left(\alpha_j+\frac{(n+\sin t)}{\sin t(1-\sin t)}\right)\left(\alpha_k+\frac{(n+\sin t)}{\sin t(1-\sin t)}\right)x_i^jx_i^k=
$$
$$
=\left(\frac{n+\sin t}{1-\sin t}\right)^2- \sum\limits_{j,k=1}^{n-1}\left(\alpha_j+\frac{(n+\sin t)}{\sin t(1-\sin t)}\right)\left(\alpha_k+\frac{(n+\sin t)}{\sin t(1-\sin t)}\right)\sum\limits_{i=1}^{n-1}x_i^jx_i^k=
$$
$$
=\left(\frac{n+\sin t}{1-\sin t}\right)^2- \frac{(n^2-1)\sin^2 t}{n^2-\sin ^2t}\left(\sum\limits_{j=1}^{n-1}\left(\alpha_j+\frac{(n+\sin t)}{\sin t(1-\sin t)}\right)^2-\right.
$$
$$
\left. -\frac{1}{n-1}\sum\limits_{j\ne k=1}^{n-1}\left(\alpha_j+\frac{(n+\sin t)}{\sin t(1-\sin t)}\right)\left(\alpha_k+\frac{(n+\sin t)}{\sin t(1-\sin t)}\right)\right)=
$$
$$
=\left(\frac{n+\sin t}{1-\sin t}\right)^2- \frac{(n^2-1)\sin^2 t}{n^2-\sin ^2t}\left(\sum\limits_{j=1}^{n-1}\left(\alpha_j^2+2\alpha_j\frac{(n+\sin t)}{\sin t(1-\sin t)}+\right.\right.
$$
$$
\left.\left.+\frac{(n+\sin t)^2}{\sin^2 t(1-\sin t)^2}\right)-\frac{1}{n-1}\sum\limits_{j\ne k=1}^{n-1}\left(\alpha_j\alpha_k+\left(\alpha_j+\alpha_k\right)\frac{(n+\sin t)}{\sin t(1-\sin t)}+\frac{(n+\sin t)^2}{\sin^2 t(1-\sin t)^2}\right)\right)=
$$
$$
=\left(\frac{n+\sin t}{1-\sin t}\right)^2- \frac{(n^2-1)\sin^2 t}{n^2-\sin ^2t}\left((n-1)-\frac{1}{n-1}(n-1)(n-2)\right)\frac{(n+\sin t)^2}{\sin^2 t(1-\sin t)^2}-
$$
$$
-\frac{(n^2-1)\sin^2 t}{n^2-\sin^2 t}\left(2\sum\limits_{j=1}^{n-1}\alpha_j-\frac{1}{n-1}\sum\limits_{j\ne k=1}^{n-1}\left(\alpha_j+\alpha_k\right)\right)\frac{(n+\sin t)}{\sin t(1-\sin t)}-
$$
$$
-\frac{(n^2-1)\sin^2 t}{n^2-\sin^2 t}\left(\sum\limits_{j=1}^{n-1}\alpha_j^2-\frac{1}{n-1}\sum\limits_{j\ne k=1}^{n-1}\alpha_j\alpha_k\right)=\frac{(n+\sin t)(1+\sin t)}{(n-\sin t)(1-\sin t)}-
$$
$$
-\frac{2(n+1)\sin t}{(n-\sin t)(1-\sin t)}\sum\limits_{j=1}^{n-1}\alpha_j-\frac{(n^2-1)\sin^2 t}{n^2-\sin ^2t} \left(\sum\limits_{j=1}^{n-1}\alpha_j^2-\frac{1}{n-1}\sum\limits_{j\ne k=1}^{n-1}\alpha_j\alpha_k\right).
$$
On the other hand obviously $\|z\|^2\geq 1$ and hence
$$
z_n^2\geq 1-\sum\limits_{i=1}^{n-1}\left(\sum\limits_{j=1}^{n-1}\alpha_j x_i^j\right)^2=1
-\frac{(n^2-1)\sin^2 t}{n^2-\sin^2 t}\left(\sum\limits_{j=1}^{n-1}\alpha_j^2 -\frac{1}{n-1}\sum\limits_{j\ne k=1}^{n-1}\alpha_j\alpha_k\right).
$$
These inequalities imply the assumption
\begin{equation}\label{assumptionzn}
1-|v|^2\leq z_n^2\leq \frac{(n+\sin t)(1+\sin t)}{(n-\sin t)(1-\sin t)}-
\frac{2(n+1)\sin t}{(n-\sin t)(1-\sin t)}\sum\limits_{j=1}^{n-1}\alpha_j-|v|^2,
\end{equation}
where $|v|^2=\frac{(n^2-1)\sin^2 t}{n^2-\sin ^2 t}\left(\sum\limits_{j=1}^{n-1}\alpha_j^2 -\frac{1}{n-1}\sum\limits_{j\ne k=1}^{n-1}\alpha_j\alpha_k\right)$ is the square of the Euclidean norm of the vector $v$.  Denote the quantity $\sum\limits_{j=1}^{n-1}\alpha_j$ by $\alpha (v)$. Then the volume is:
$$
V(\tau[n,t])=\sum\limits_{i=1}^{n}\int\limits_{v\in \epsilon_i[n-1,t]}\int\limits_{z_n=\sqrt{1-|v|^2}} ^{\sqrt{\frac{(n+\sin t)(1+\sin t)}{(n-\sin t)(1-\sin t)}-
\frac{2(n+1)\sin t}{(n-\sin t)(1-\sin t)}\alpha (v)-|v|^2}}z_n^{-n}\mathrm{d}z_n\mathrm{d}\rho.
$$
Since the set of the vectors $v$ belonging to more than one $\epsilon_i[n-1,t]$, has measure zero, the above equality can be written of the form:
\begin{equation}\label{pvolume}
V(\tau[n,t])=\frac{1}{n-1}\int\limits_{\epsilon[n-1,t]}\left(\frac{1}{\sqrt{1-\rho^2}}\right)^{n-1}\mathrm{d}\rho-
$$
$$
-\frac{1}{n-1}\int\limits_{\epsilon[n-1,t]}{\left(\frac{1}{\sqrt{\frac{(n+\sin t)(1+\sin t)}{(n-\sin t)(1-\sin t)}-
\frac{2(n+1)\sin t}{(n-\sin t)(1-\sin t)}\alpha(v)-\rho^2}}\right)}^{n-1}\mathrm{d}\rho,
\end{equation}
where $\rho^2=|v|^2$ and $\alpha(v)$ is uniquely determined by the vector $v$.

By formula (\ref{pvolume}) we can get an upper bound on the volume of $\tau[n,t]$. 
\begin{lemma}\label{lemmaupperb}
\begin{equation}\label{finalupperb}
V(\tau[n,t])\leq \frac{1}{n-1}\left(1-\left(\sqrt{\frac{n^2(1-\sin t)^2(1+\sin t)}{(n+\sin t)^2(1+\sin t)-(n^2-1)\sin ^2 t(1-\sin t)^2}}\right)^{n-1}\right)V(\tau_i[n,t]).
\end{equation}
\end{lemma}

\begin{proof}
Clearly
$$
\sqrt{\frac{(n+\sin t)(1+\sin t)}{(n-\sin t)(1-\sin t)}-
\frac{2(n+1)\sin t}{(n-\sin t)(1-\sin t)}\alpha(v)-|v|^2}\leq \sqrt{\frac{(n+\sin t)(1+\sin t)}{(n-\sin t)(1-\sin t)}-|v|^2}
$$
and using the fact that if $\zeta> 1$ the function $f(x)=(\zeta -x)/(1-x)$ is strictly increasing for positive $x$, we get 
$$
\sqrt{\frac{(n+\sin t)(1+\sin t)}{(n-\sin t)(1-\sin t)}-|v|^2}\leq \sqrt{\frac{\frac{(n+\sin t)(1+\sin t)}{(n-\sin t)(1-\sin t)}-|v|^2_{max}}{1-|v|^2_{max}}}\left(\sqrt{1-|v|^2}\right).
$$
Since $v\in \epsilon[n-1,t]$ we have $|v|^2_{max}= h(r_{n-1})^2 =\left(\frac{\sqrt{n^2-1}\sin t}{\sqrt{n^2-\sin ^2t}}\right)^2$.
Since the radius of the circumscribed sphere of $\epsilon[n-1,t]$ is equal to $\sin\alpha=\frac{\sqrt{n^2-1}\sin t}{\sqrt{n^2-\sin ^2t}}$ we get 
$$
V(\tau[n,t])\leq \frac{1}{n-1}\left(1-\left(\sqrt{\frac{1-|v_{max}|^2}{\frac{(n+\sin t)(1+\sin t)}{(n-\sin t)(1-\sin t)}-|v_{max}|^2}}\right)^{n-1}\right)\int\limits_{\frac{\sqrt{n^2-1}\sin t}{\sqrt{n^2-\sin ^2t}}S(n-1)}\left(\frac{1}{\sqrt{1-|v|^2}}\right)^{n-1}\mathrm{d}|v|\leq
$$
$$
\leq \frac{1}{n-1}\left(1-\left(\sqrt{\frac{n^2(1-\sin t)^2(1+\sin t)}{(n+\sin t)^2(1+\sin t)-(n^2-1)\sin ^2 t(1-\sin t)^2}}\right)^{n-1}\right)\times
$$
$$
\times \int\limits_{\frac{\sqrt{n^2-1}\sin t}{\sqrt{n^2-\sin ^2t}}S(n-1)}\left(\frac{1}{\sqrt{1-|v|^2}}\right)^{n}\mathrm{d}|v|,
$$
and using (\ref{ckvolumeofthefacets}) we get
$$
V(\tau[n,t])\leq \frac{1}{n-1}\left(1-\left(\sqrt{\frac{n^2(1-\sin t)^2(1+\sin t)}{(n+\sin t)^2(1+\sin t)-(n^2-1)\sin ^2 t(1-\sin t)^2}}\right)^{n-1}\right)V(\tau_i[n,t]),
$$
as we stated.
\end{proof}

\begin{remark}\label{remgenvolume} An analogue of (\ref{pvolume}) can be determined in such a case when the simplex $\mathcal{F}$ is not regular. If we assume that the circumcenter $h^{-1}(\widetilde{O})$, the vertex $h^{-1}(\widetilde{E_{n+1}})$ and the center $h^{-1}(\widetilde{K_{n-1}})$ of the facet $\mathcal{F}_{n+1}$ are collinear and $h^{-1}(\widetilde{K_{n-1}})$ is an inner point of $\mathcal{F}_{n+1}$ then the half-space model representation of $\mathcal{F}$ is very similar to the regular case. We have to take distinction only at those points in the calculation when (firstly) we determined and used the Gramm matrix $R(n-1)$ of the regular vector system, and (secondly) when we determined the distances of the points $\widetilde{K_{n-1}}$,$\widetilde{O}$ and $\widetilde{E_{n+1}}$, respectively. Let denote $\rho(h^{-1}(\widetilde{K_{n-1}}),h^{-1}(\widetilde{O}))$ by $d$ and $\rho(h^{-1}(\widetilde{O}),h^{-1}(\widetilde{E_{n+1}}))$ by $r$, respectively. As in the regular case we can determine the coordinate $z_n$ of a point $z=v+z_ne_n, \quad v\in \epsilon_i[n-1,t]$ . It is bounded above with the value $(e^{r+d}+1)^2-e^{r+d}(e^{r+d}+2)\alpha(v)-|v|^2$ where $\alpha(v)$ is the same quantity as in the regular case. We denote by $\varepsilon[n-1]$ the orthogonal projection of $h(\mathcal{F}_{n+1})$ to $H$. The general formula now is
\begin{equation}\label{genpvolume}
V(\mathcal{F})=\frac{1}{n-1}\int\limits_{\varepsilon[n-1]}\left(\left(\frac{1}{\sqrt{1-\rho^2}}\right)^{n-1}- \left(\frac{1}{\sqrt{(e^{r+d}+1)^2-e^{r+d}(e^{r+d}+2)\alpha(v)-\rho^2}}\right)^{n-1}\right)\mathrm{d}\rho.
\end{equation}
\end{remark}
From this formula immediately follows an upper bound on the volume is similar to (\ref{finalupperb}).


\begin{thebibliography}{99}

\bibitem{bolyai} Bolyai, J. {\em Appendix} in Tentamen written by F.Bolyai, Marosv\'as\'arhely, 1832

\bibitem{cho} Cho, Y., Kim, H., On the Volume Formula for Hyperbolic Tetrahedra {\em Discrete Comput. Geom.} {\bf 22} 1999, 347–366.

\bibitem{mednykh} Derevnin, D. A., Mednykh, A.D., A formula for the volume of a hyperbolic tetrahedon {\em Uspekhi Mat. Nauk} {\bf 60/2} 2005, 159-160.

\bibitem{gho 1} G. Horv\'ath, \'A., Formulas on hyperbolic volume. {\em  Aequationes Mathematicae} {\bf 83/1} (2012), 97-116.

\bibitem{gho 2} G. Horv\'ath, \'A., Malfatti's problem on the hyperbolic plane. {\em Studia Sci. Math.} {\bf 51/2} (2014), 201--212.

\bibitem{gho 3} G. Horv\'ath, \'A., Hyperbolic plane geometry revisited. {\em Journal of Geometry}  {\bf 106/2} (2015), 341-362.

\bibitem{gho 4} G. Horv\'ath, \'A., On the hyperbolic triangle centers {\em Studies of the University of Zilina} {\bf 27/1} (2015) 11--34.

\bibitem{haagerupmunkholm} Haagerup, U., Munkholm, H. J., Simplices of maximal volume in hyperbolic n-space. {\em Acta. Math.} {\bf 147} (1981), 1- 12.

\bibitem{kellerhals} Kellerhals, R., On the volume of hyperbolic polyhedra {\em Math. Ann.} {\bf 285},(1989) 541-569.

\bibitem{lobacsevszkij}Lobachevsky, N.I. {\em Zwei Geometrische Abhandlungen} B.G.Teubner, Leipzig and Berlin, 1898, (reprinted by Johnson Reprint Corp., New York and London, 1972)

\bibitem{milnor} Milnor, J.W., {\em Computation of volume}. Lecture at Princton University. See section 7 of \cite{thurston}.

\bibitem{milnor 2} Milnor, J. W. Hyperbolic geometry: The first 150 years. {\em Bull. Amer. Math. Soc. (N.S.)} {\bf 6} (1982), no. 1, 9--24. 

\bibitem{mitrinovic} Mitrinovic, D.S., Pecaric, J.E., Fink, A.M.; {\em Classical and New Inequalities in Analysis} Kluwer Academic Publisher, Dordrecht, Boston, London (1993).

\bibitem{murakami}Murakami, J., Yano, M., On the volume of  Hyperbolic and Spherical tetrahedron,{\em Communications in
analysis and geometry} {\bf 13/2}, (2005) 379-400.

\bibitem{thurston} Thurston, W.P., {\em The Geometry and Topology of $3$-manifolds.} Lecture notes from Princton University 1977/78.
\end{thebibliography}
\end{document}